\documentclass{amsart}

\usepackage{amssymb}
\usepackage{amsthm}
\usepackage{newlfont}
\usepackage{amsfonts}
\usepackage{graphicx}
\usepackage[all]{xy}
\usepackage{subfigure}

\theoremstyle{theorem}
\newtheorem{theorem}{Theorem}
\newtheorem{proposition}{Proposition}
\newtheorem{lemma}{Lemma}

\theoremstyle{definition}

\newtheorem{remark}{Remark}
\newtheorem{question}{Question}

\begin{document}

\title[MINIMAL SURFACE WITH TWO CATENOID
ENDS AND ONE ENNEPER END]{AN EXAMPLE OF A MINIMAL SURFACE OF GENUS ONE WITH TWO CATENOID 
ENDS AND ONE ENNEPER END}


\author{Jos\'{e} Antonio M. Vilhena}
\address{Universidade Federal do Par\'a\\ Instituto de Ci\^encias Exatas e Naturais\\ 66075-110 \\ Bel\'{e}m - PA \\ BRASIL.}
\address{Jos\'{e} Antonio M. Vilhena}
\email{vilhena@ufpa.br}
\thanks{FACMAT - ICEN - UFPA}


\date{December 23, 2019}

\maketitle

\begin{abstract}
In this paper we construct an example  of a complete immersed minimal surface in $\mathbb{R}^3$ of genus one with two embedded catenoid-type ends, one Enneper-type end and total Gauss curvature $-16\pi.$ The proof of the existence of this example, was obtained using the Weierstrass representation, the theory of elliptic functions and explicitly solving the period problem.
\end{abstract}

\bigskip
\noindent

\noindent   

\section{Introduction}

At the beginning of the 80's and 90's, the Weierstrass representation and the theory of elliptic functions was a fundamental tool for finding a large quantity of new examples of minimal surfaces: Chen-Gackstatter, Costa, Hoffman, Meeks and Karcher (see \cite{Chen.1982}, \cite{Costa.1989}, \cite{Hoffman.1990},\cite{Hoffman.1985}) among others. 

The Chen-Gackstatter surface (C--G) (see \cite{Chen.1982}) was the first example of a complete minimal surface of genus one with one Enneper-type end and total curvature $-8\pi.$  The construction of this example  was obtained using  the elliptic function $\wp$ and  the Weierstrass data $\displaystyle{(g, \mathbf \eta)=(c\, \wp'/\wp,2\wp \,dz), \ c=\sqrt{3\pi/2g_2}.}$ In the same paper, Chen-Gackstatter have proven that there is a complete minimal surface with genus two and and one Enneper-type end. The genus one C-G surface was generalized by Karcher \cite{Karcher.1989} and the genus two C-G surface was generalized by Thayer \cite{Thayer.1995}. These generalizations are similar to C-G surfaces, but with higher winding order at the end. Other generalizations of the Chen-Gackstatter surface also can be found in \cite{Nedir.1994}, \cite{Fang.1990}, \cite{Kang.2003}. Among many other surfaces of genus one with three ends, we can mention the beautiful Costa surface (see \cite{Costa.1984}) given by Weierstrass data $\displaystyle{(g, \mathbf \eta)=(a/\wp',\wp \,dz), \ a=2e_1\sqrt{2\pi}}$ that is an embedded minimal surface with one planar end and two catenoid ends.

In this paper we first give a description of the Matthias Weber's minimal surface \cite{Weber.2015} and using the theory of elliptic function, we explicitly solve the period problem. The main goal of this paper, will be to prove that there exists a complete minimal surface of genus one, with two parallel catenoid ends and one Enneper end, explicitly solving the period problem. 

\begin{theorem}\label{Teorema1}
There exists a complete minimal immersion $S$ in $\mathbb{R}^3$, of genus
one, with three ends and the following properties:
\begin{enumerate}
	\item The total curvature of $S$ is $-16 \pi$;
	\item $S$ has two catenoid-type ends and one Enneper-type end;
	\item The symmetry group of $S$ is the dihedral group $G$ with $8$ elements generated by
	\[
	A_{\beta}=\left[
	\begin{array}{ccc}
	1&0&0\\
	0&-1&0\\
	0&0&1
	\end{array}
	\right], \ \ \ 
	A_{\rho}=\left[
	\begin{array}{ccc}
	0&-1&0\\
	1&0&0\\
	0&0&-1
	\end{array}
	\right];
	\]
	\item The Weierstrass data is given by
	\[
	\left
	\{
	\begin{array}{ll}
	 g &= c \, \displaystyle{\frac{(\wp - 3 e_1)(\wp+3e_1)}{\wp'}, \quad c=\frac{1}{e_1}\sqrt{\frac{6\pi}{73}}} \, , \ e_1=\wp(1/2);\\
	\mathbf \eta&=2\wp \, dz.
	\end{array}
	\right.
	\]
\end{enumerate}
\end{theorem}

\begin{figure}[h]
\subfigure[\label{Figure_a}]{
\includegraphics[totalheight=6.3cm]{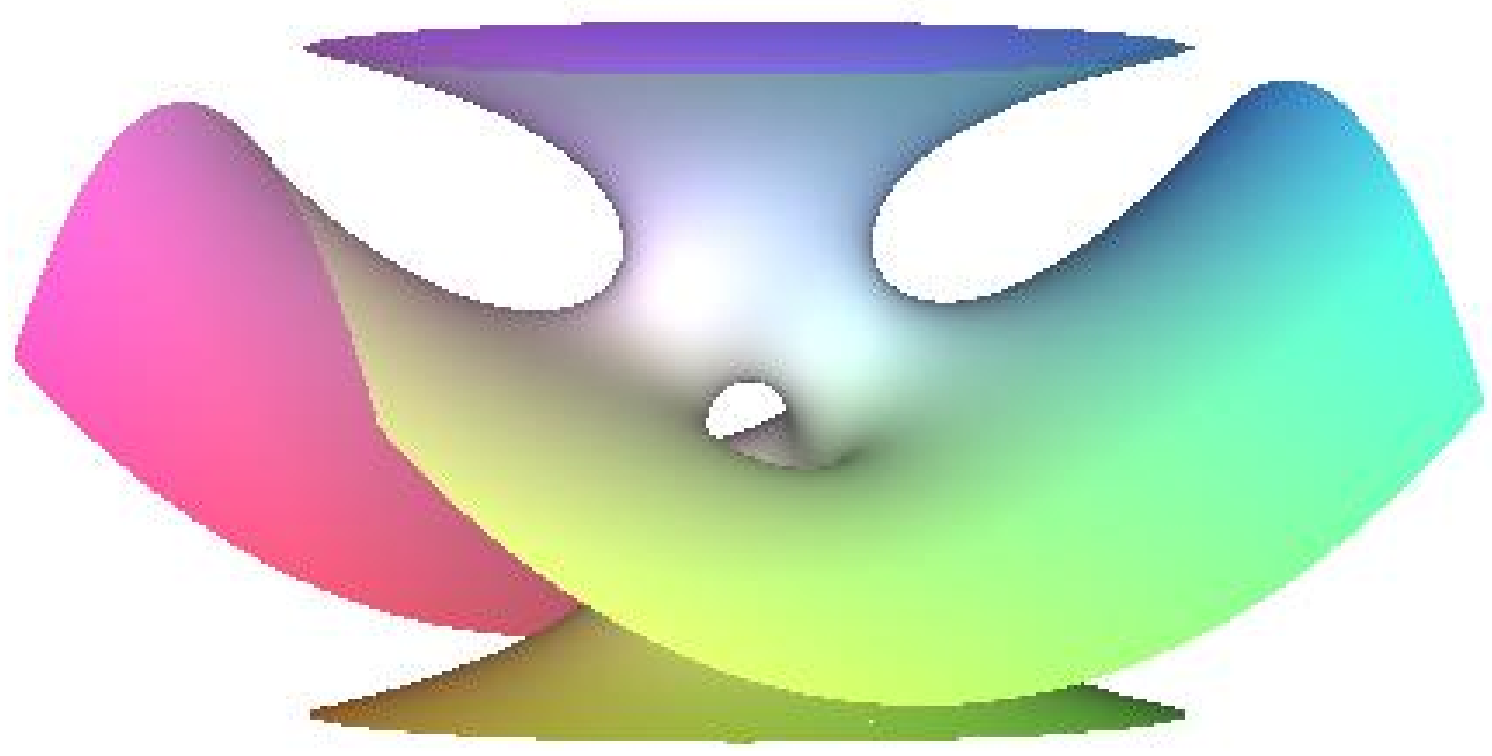}}
\subfigure[\label{Figure_b}]{
\includegraphics[totalheight=6cm]{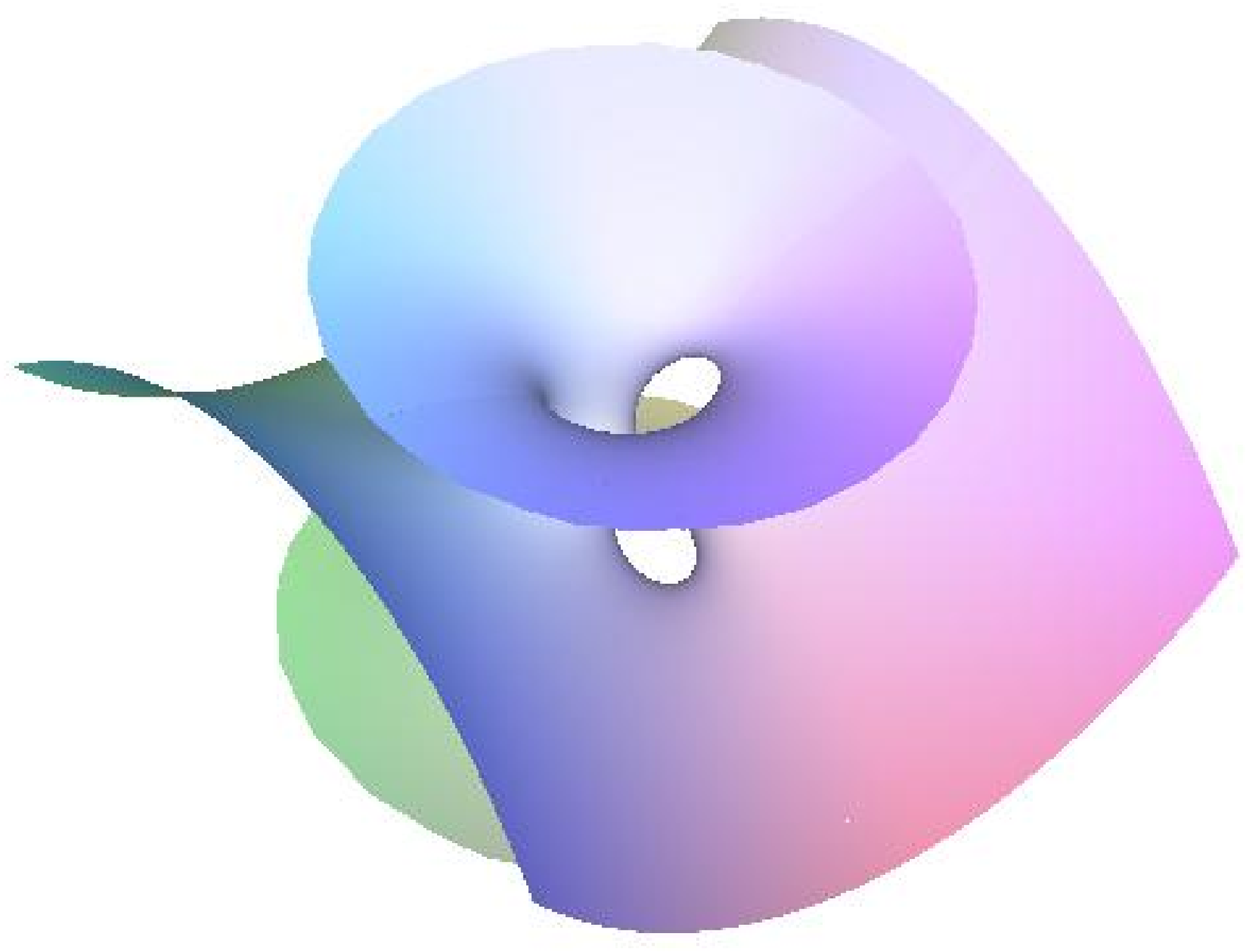}}
\caption{Computer graphics of $S$}
\end{figure}

\section{Preliminaries}

\subsection{The Weierstrass representation}

The main tool we will use to prove the existence of the surface described in Theorem \ref{Teorema1} is the Weierstrass representation formula (see  \cite{Fang.1990}, \cite{Osserman.2013}).

\begin{proposition}\label{Prop1} Let $\overline{M}$ be a compact Riemann surface and $M =\overline{M}-\{p_1, \cdots, p_n\}.$ Suppose $\overline{g}\!: \overline{M} \rightarrow \mathbb{C}\cup \{\infty\} $ is a meromorphic function and $\eta$ is a meromorphic $1$-form such that whenever $g=\overline{g}|_M$ has a pole of order $k$, then $\eta$ has a zero of order $2k$ and $\eta$ has no other zeros on $M$. Let
\begin{equation}\label{WR}
\phi_1=\frac{1}{2}\left(1-g^2 \right)\eta , \  \phi_2=\frac{i}{2}\left(1+g^2 \right)\eta,  \ \phi_3=g \eta. 
\end{equation}
If for any closed curve $\gamma$ in $M$,
\begin{equation}\label{PP1}
\text{Re} \int_{\gamma} \phi_j=0, \ \text{for} \ \  j =1,2,3,
\end{equation}
and every divergent curve $\ell$ in $M$ has infinity length, i.e.,
\begin{equation}\label{PS0}
\int_{\ell} (1+|g|^2)|\eta|=\infty,
\end{equation}
then the surface $S$, defined by $X\!:M \rightarrow \mathbb{R}^3$, is a complete regular minimal surface, where
\begin{equation}\label{PS}
X(z)=\text{Re}\left(\int_{z_0}^z \phi_1, \int_{z_0}^z \phi_2, \int_{z_0}^z \phi_3\right).
\end{equation}
Here, $z_0$ is a fixed point of $M$. Moreover, 
the total curvature of $S$ is
\begin{equation}\label{CT}
C_T(S) = -4 \pi m,
\end{equation}
where $m$ is the degree of $\overline{g}.$
\end{proposition}
\subsection{The Weierstrass $\wp$ and $\zeta$ functions}
For the proof of Theorem \ref{Teorema1}, we get some properties of the Weierstrass
$\wp$ function associated to lattice $\mathcal{L} = [ 1, i]$. The proofs of the results below can be found in  \cite{Gray.2017}, \cite{Chand.1985} and  \cite{Costa.1984}.
Let $\mathcal{L} = [w_1,w_2]$, $\text{Im}\displaystyle{\left(\frac{w_2}{w_1}\right)}>0$ be a lattice of $\mathbb{C}$. The Weierstrass $\wp$ function of the lattice $\mathcal{L}$ is a doubly
periodic meromorphic function, defined by 
\begin{equation}\label{PW}
\wp(z)=\frac{1}{z^2}+\sum_{\substack{\Omega \in \mathcal{L} \\ \Omega \neq 0}}\left(\frac{1}{(z-\Omega)^2}-\frac{1}{\Omega^2} \right), 
\end{equation}
where $\Omega=m w_1+n w_2,$ \ for all $(m,n) \in \mathbb{Z} \times \mathbb{Z}$ and $(m, n ) \neq (0,0)$.

We will also need the Weierstrass $\zeta$ functions, defined by

\begin{equation}\label{ZW}
\zeta(z)=\frac{1}{z}+\sum_{\substack{\Omega \in \mathcal{L} \\ \Omega \neq 0}}\left(\frac{1}{z-\Omega}+\frac{1}{\Omega}+\frac{z}{\Omega^2} \right). 
\end{equation}
From \eqref{PW} and \eqref{ZW} we have that $\zeta$ is related to $\wp$ by 
\begin{equation}\label{ZP}
\zeta(z)'=-\wp(z).
\end{equation}

It is possible to express $\wp(z+z_1)$ in terms of $\wp(z)$,  and $\wp(z_1)$ and their derivatives. 

\begin{proposition}
If $z \neq z_1$ modulo $(w_1, w_2),$ then we have
	\begin{equation}
\wp(z+z_1)=\frac{1}{4}\left(\frac{\wp'(z)-\wp'(z_1)}{\wp(z)-\wp(z_1)} \right)^2 - \wp(z) - \wp(z_1).
\end{equation}
\end{proposition}

\begin{proposition}
The elliptic function $\wp(z)$ satisfies the differential equation
\begin{equation}
\wp'(z)^2= 4 \wp(z)^3 - g_2 \wp(z) - g_3,
\end{equation}
where 
\begin{equation}
g_2 = 60 \sum_{\substack{\Omega \in \mathcal{L} \\ \Omega \neq 0}}   \Omega^{-4} , \ \ 
g_3 = 140 \sum_{\substack{\Omega \in \mathcal{L} \\ \Omega \neq 0}}   \Omega^{-6}. 
\end{equation}
\end{proposition}

The three zeros of $\wp'(z)^2=4 \wp(z)^3 - g_2 \wp(z) - g_3$ in $\mathcal{L}=\mathcal{L}[1,i]$ are $w_1=1/2, \, w_2=(1+i)/2$ and $w_3=i/2$ and putting $e_1:=\wp(1/2),$ $e_2:=\wp((1+i)/2),$ $e_3:=\wp(i/2),$ we have that
\begin{equation}
\wp'(z)^2= 4 \left( \wp - e_1 \right) \left( \wp - e_2 \right) \left( \wp - e_3 \right), \  \  \wp=\wp(z),
\end{equation}
with
\begin{equation}\label{e123}
	e_1+e_2+e_3=0, \ 
	e_1e_2+e_2e_3+e_1e_3=-\displaystyle{\frac{g_2}{4}}, \ 
	e_1e_2e_3=\displaystyle{\frac{g_3}{4}}.
\end{equation}
From now on, we assume that $\mathcal{L}=\mathcal{L}[1,i]$  and that the fundamental domain is $F=\{ z \in \mathbb{C} \ | \ 0 \leqslant \text{Re}(z)  < 1, \   0 \leqslant \text{Im}(z) < 1\}.$ 
In this case, the results below are well known.
 
\begin{lemma}\label{lemma1}
Let $\mathcal{L} = \mathcal{L}[1, i]$ be a lattice and $z \in F$. Then, with the above notation, we have:
\begin{equation}\label{e2}
	e_2=g_3=0,
\end{equation}
\begin{equation}\label{e3}
	e_3=-e_1,
\end{equation}
\begin{equation}\label{g2}
	g_2=4e_1^2,
\end{equation}
\begin{equation}\label{EDPW}
	\wp'(z)^2= 4 \wp \left( \wp - e_1 \right) \left( \wp +e_1 \right),
\end{equation}
\begin{equation}\label{EP2}
\wp^2=\frac{\wp''}{6}+\frac{e_1^2}{3}.
\end{equation}
\end{lemma}

\begin{lemma}\label{lemma2}
Let $\mathcal{L}$ be a lattice and $z \in F$. Then, we have:
\begin{equation}\label{pe1}
\frac{1}{\wp-e_1}=\frac{\wp(z-1/2)}{2e_1^2}-\frac{e_1}{2e_1^2},
\end{equation}
\begin{equation}\label{pe3}
\frac{1}{\wp+e_1}=\frac{\wp(z-i/2)}{2e_1^2}+\frac{e_1}{2e_1^2},
\end{equation}
\begin{equation}\label{EF}
\frac{1}{(\wp-e_1)(\wp+e_1)}=\frac{1}{2e_1}\cdot \frac{1}{\wp-e_1}-\frac{1}{2e_1}\cdot \frac{1}{\wp+e_1},
\end{equation}
\end{lemma}

\begin{lemma}\label{lemma3}
Let $\mathcal{L}$ be a lattice and $z \in F$. Then, we have:
\begin{equation}\label{EqW1}
\frac{1}{\wp-e_1}-\frac{1}{\wp+e_1}=\frac{1}{2e_1^2}\left(\wp(z-1/2)-\wp(z-i/2)-2e_1\right),
\end{equation}

\begin{equation}\label{EqW2}
\frac{\wp}{\wp-e_1}-\frac{\wp}{\wp+e_1}=\frac{1}{2e_1}\left(\wp(z-1/2)+\wp(z-i/2)\right),
\end{equation}

\begin{equation}\label{EqW3}
\frac{\wp^2}{\wp-e_1}-\frac{\wp^2}{\wp+e_1}=e_1+\frac{1}{2}\left(\wp(z-1/2)-\wp(z-i/2)\right), 
\end{equation}

\begin{equation}\label{EqW4}
\frac{\wp^3}{\wp-e_1}-\frac{\wp^3}{\wp+e_1}=2e_1\wp+\frac{e_1}{2}\left(\wp(z-1/2)+\wp(z-i/2)\right),  
\end{equation}

\begin{equation}\label{EqW5}
\frac{\wp^4}{\wp-e_1}-\frac{\wp^4}{\wp+e_1}=2e_1\wp^2+e_1^3+\frac{e_1^2}{2}\left(\wp(z-1/2)-\wp(z-i/2)\right), 
\end{equation}

\begin{equation}\label{EqW6}
\frac{\wp' \wp}{\wp-e_1}-\frac{\wp' \wp}{\wp+e_1}=e_1\frac{\wp'}{\wp-e_1}+e_1\frac{\wp'}{\wp+e_1},
 \end{equation}
\begin{equation}\label{EqW7}
\frac{\wp' \wp^2}{\wp-e_1}-\frac{\wp' \wp^2}{\wp+e_1}=2e_1\wp'+e_1^2\frac{\wp'}{\wp-e_1}-e_1^2\frac{\wp'}{\wp+e_1}.
\end{equation}


\end{lemma}
\begin{proof}
From \eqref{pe1} and \eqref{pe3}, it follows  \eqref{EqW1}. To prove item \eqref{EqW3}, observe that
\begin{equation}\label{pw1}
\frac{\wp^2}{\wp-e_1}=\wp+e_1+\frac{e_1^2}{\wp-e_1},
 \end{equation}
\begin{equation*}
\frac{\wp^2}{\wp+e_1}=\wp-e_1+\frac{e_1^2}{\wp+e_1}.
\end{equation*}
From this it follows that
\begin{equation*}
\frac{\wp^2}{\wp-e_1}-\frac{\wp^2}{\wp+e_1}=2e_1+e_1^2\left(\frac{1}{\wp-e_1}-\frac{1}{\wp+e_1}\right). 
\end{equation*}
Using \eqref{EqW1}, the result follows. In an analogous way,  we can prove the other formulae.
\end{proof}

Let $\alpha, \beta \!: [0,1] \rightarrow \mathbb{C}$ be the paths
\begin{equation}\label{basisH}
\alpha(t) =\frac{i}{3}+t, \ \beta(t)=\frac{1}{3}+it.
\end{equation}
The set $\{\alpha, \beta\}$ is a non-trivial, homology basis of the torus $\displaystyle{T^2=\mathbb{C}/\mathcal{L}[1,i]}.$

\begin{lemma}\label{lemma4}
If $\mathcal{L} = [1,i]$, \ $z, z_0 \in F,$ here $z_0$ is a fixed point of $F$, then

\begin{equation}
	\int_{z_0}^z \frac{1}{\wp-e_1} \, dz=-\frac{1}{2e_1^2} \zeta(z-1/2)-\frac{1}{2e_1}z+c_1,
\end{equation}
	
\begin{equation}
	\int_{z_0}^z \frac{\wp}{\wp-e_1} \, dz=\frac{z}{2} - \frac{1}{2e_1}\zeta(z-1/2)+c_2,
\end{equation}

\begin{equation}
\int_{z_0}^z \frac{\wp^2}{\wp-e_1} \, dz=\frac{e_1}{2}z-\zeta(z)-\frac{1}{2}\zeta(z-1/2)+c_3, 
\end{equation}

\begin{equation}\label{intp3}
\int_{z_0}^z \frac{\wp^3}{\wp-e_1} \, dz=\frac{5e_1^2}{6}z+\frac{1}{6}\wp'-e_1\zeta(z)-\frac{e_1}{2}\zeta(z-1/2)+c_4. 
\end{equation}
where $c_j, \ j=1,2,3,4,5$ are constants.
\end{lemma}
\begin{proof}
To prove item \eqref{intp3}, observe that
\begin{equation*}
\frac{\wp^3}{\wp-e_1}=\wp^2+e_1 \, \frac{\wp^2}{\wp-e_1},
 \end{equation*}
Using \eqref{EP2}, \eqref{pe1} and \eqref{pw1}, we have
\begin{equation*}
\frac{\wp^3}{\wp-e_1}=\frac{5}{6}e_1^2+\frac{1}{6}\wp''+e_1 \wp +\frac{e_1}{2}\wp(z-1/2).
 \end{equation*}
Hence
\begin{equation*}
\int_{z_0}^{z}\frac{\wp^3}{\wp-e_1}=\frac{5}{6}e_1^2z+\frac{1}{6}\wp'-e_1 \zeta(z) -\frac{e_1}{2}\zeta(z-1/2)+c_4.
\end{equation*}
In an analogous way,  we can prove the other cases.
\end{proof}

\begin{lemma}\label{lemma5}
If $\mathcal{L} = [1,i]$, \ $z, z_0 \in F,$ here $z_0$ is a fixed point of $F$, then

\begin{equation}
	\int_{z_0}^z \left(\frac{1}{\wp-e_1}-\frac{1}{\wp+e_1}\right) \, dz=\frac{1}{2e_1^2}\Big(-\zeta(z-1/2)+\zeta(z-i/2)-2e_1z\Big)+c_1,
\end{equation}
	
\begin{equation}
	\int_{z_0}^z \left(\frac{\wp}{\wp-e_1}-\frac{\wp}{\wp+e_1}\right) \, dz=-\frac{1}{2e_1}\zeta(z-1/2)-\frac{1}{2e_1}\zeta(z-i/2)+c_2,
\end{equation}

\begin{equation}
\int_{z_0}^z \left(\frac{\wp^2}{\wp-e_1}-\frac{\wp^2}{\wp+e_1}\right) \, dz=e_1z-\frac{1}{2}\zeta(z-1/2)+\frac{1}{2}\zeta(z-i/2)+c_3, 
\end{equation}

\begin{equation}
\int_{z_0}^z \left(\frac{\wp^3}{\wp-e_1}-\frac{\wp^3}{\wp+e_1}\right) \, dz=-2e_1\zeta(z)-\frac{e_1}{2}\zeta(z-1/2)-\frac{e_1}{2}\wp(z-i/2)+c_4, 
\end{equation}

\begin{equation}
\int_{z_0}^z \left(\frac{\wp^4}{\wp-e_1}-\frac{\wp^4}{\wp+e_1}\right) \, dz=\frac{e_1}{3}\wp'(z)+\frac{5e_1^3}{3}z-\frac{e_1^2}{2}\zeta(z-1/2)+\frac{e_1^2}{2}\zeta(z-i/2)+c_5,
\end{equation}
where $c_j, \ j=1,2,3,4,5$ are constants.
\end{lemma}

Using the Legendre's relation we have 

\begin{lemma}(\cite{Costa.1984})\label{lemma6}
Let $\alpha, \beta \!: [0,1] \rightarrow \mathbb{C}$ be the paths
\begin{equation*}
\alpha(t) =\frac{i}{3}+t, \ \beta(t)=\frac{1}{3}+it
\end{equation*}
of homology basis of the torus $\displaystyle{T^2=\mathbb{C}/\mathcal{L}}.$ Then,

\begin{equation}\label{intP}
\int \wp(z) \, dz = - \zeta(z) + const.,
\end{equation}

\begin{equation}
\zeta(1/2)=\frac{\pi}{2}, \ \zeta(i/2)=-\frac{\pi}{2}i, \  \zeta\left(\frac{1+i}{2}\right)=\frac{\pi}{2}-\frac{\pi}{2}i,
\end{equation}

\begin{equation}\label{intA}
\int_{\alpha} \wp(z) \, dz = \int_{\alpha} \wp(z-1/2) \, dz = \int_{\alpha} \wp(z-i/2) \, dz=  \int_{\alpha} \wp(z-(1+i)/2) \, dz = -\pi,
\end{equation}

\begin{equation}\label{intB}
\int_{\beta} \wp(z) \, dz = \int_{\beta} \wp(z-1/2) \, dz = \int_{\beta} \wp(z-i/2) \, dz=  \int_{\beta} \wp(z-(1+i)/2) \, dz = i\pi.
\end{equation}

\end{lemma}

\subsection{Symmetries of the Weierstrass $\wp$ functions}
The symmetries of the minimal surface $S=X(M)$ given in Theorem \ref{Teorema1} are  consequence of the symmetries of the Weierstrass $\wp$ functions no fundamental domain $F$ and of Proposition \ref{Wohlgemuth.1991} below.

\begin{lemma}\label{LemaS1}(\cite{Hoffman.1985})
Let $\wp(z)$ be the Weierstrass $\wp$-function for the unit-square lattices:
\begin{enumerate}
\item $\wp(\rho(w_2+z))=-\wp(w_2+z), \ \ \rho(w_2+z)=w_2+iz,$
\item $\wp(\beta(w_2+z))=\overline{\wp(w_2+z)}, \ \ \beta(w_2+z)=w_2+\overline{z},$
\item $\wp(\rho\circ\beta(w_2+z))=-\overline{\wp(w_2+z)},$
\item $\wp(\rho^2\circ\beta(w_2+z))=\overline{\wp(w_2+z)},$
\item $\wp(\mu(w_2+z))=-\overline{\wp(w_2+z)}, \ \ \mu(w_2+z)=w_2-i\overline{z}. $
\end{enumerate}
\end{lemma}
\begin{remark}
We note that $\rho$, $\beta$, $\rho^2 \circ \beta$, $\rho \circ \beta$ and $\mu$ above, are respectively, a rotation by $\pi/2$ about $w_2$, reflection about the horizontal line,  reflection about the vertical line, reflection about the positive diagonal and reflection about the negative diagonal.
\end{remark}

\begin{lemma}(\cite{Karcher.1989})\label{Karcher.1989}
Let $\zeta$ be the curve in $M$ such that $g \circ \zeta$ is either a meridian of $\mathbb{S}^2$ or $g \circ \zeta$ is the equator of $\mathbb{S}^2$ and $(g\eta)(\zeta')$ is real or imaginary. Then
$\zeta$ is a geodesic on $M$.
\end{lemma}

\begin{proposition}\label{Wohlgemuth.1991}
(\cite{Wohlgemuth.1991}) If the second fundamental form $II = dg \cdot\eta$ is real when applied to
the tangent vector field of a geodesic, that geodesic is a planar curve of symmetry
and if $dg \cdot\eta$ is imaginary, it is a line.
\end{proposition}

\noindent {\bf The Schwarz Reflection Principle for Minimal Surfaces} (for more details see \cite{Fujimoto.2013}, \cite{Karcher.1989}) 

{\it If a minimal surface contains a line segment $L$, then it is symmetric under
rotation by $\pi$ about $L$.}

{\it If a nonplanar minimal surface contains a principal geodesic - necessarily
a planar curve - then it is symmetric under reflection in the plane of that
curve.}
\section{Explicit solution to period problem for Weber's minimal surface}

The Professor Matthias Weber \cite{Weber.2015}(see also \cite{Fujimori.2016}) has constructed numerically an example of a minimal surface of genus one, with one catenoid-type end and one Enneper-type end. In the Proposition \ref{S_VW} below, using the theory of elliptic functions, we will give a description of the Matthias Weber's minimal surface and give an explicit proof of the solution of the period problem.

\begin{proposition}\label{S_VW}
There exists a complete minimal immersion $S$ in $\mathbb{R}^3$, of genus one, with two ends and the following properties:
\begin{enumerate}
	\item The total curvature of $S$ is $-12 \pi$;
	\item $S$ has one catenoid-type end and one Enneper-type end;
\end{enumerate}
\end{proposition}
\begin{proof}
Let $T^2=\mathbb{C}/\mathcal{L}$ be the o torus with complex structure induce by canonical projection $\pi: \mathcal{C} \rightarrow T^2.$ 
Let $M=T^2 -\{p_1, p_2\}$, where
\begin{equation}
p_1=\pi(1/2),\ p_2=\pi(0).
\end{equation}
The Weierstrass data  $(g, \eta)$ is given by

\begin{equation}\label{WR-WW}
\left\{
\begin{array}{ll}
	 g &= c \, \displaystyle{\frac{(\wp + e_1)(\wp-\lambda)}{\wp'}},  \ \ c >0, \ \lambda> 0, \ \lambda \neq e_1,\\
	\mathbf \eta&=2\wp \, dz.
	\end{array}
	\right.
\end{equation}

The Figure \ref{Figura1} below,  shows the zeros and poles of $g$, $\eta$ and $\phi_3$ on $F.$
\begin{figure}[h]
\subfigure[Zeros and polos of $g$]{
\centering
\includegraphics[totalheight=3cm]{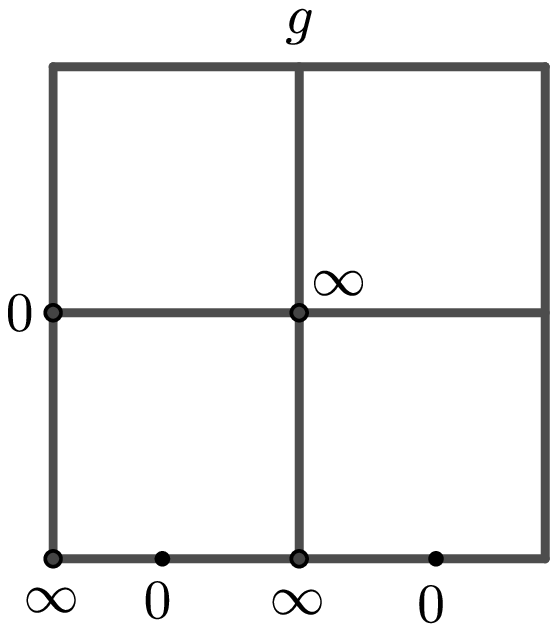}}
\subfigure[Zeros and polos of $\eta$]{
\includegraphics[totalheight=3cm]{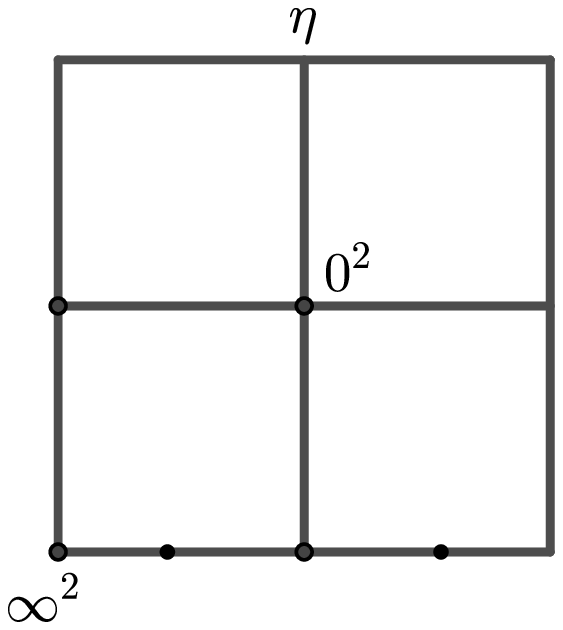}}
\subfigure[Zeros and polos of $\phi_3$]{
\includegraphics[totalheight=3cm]{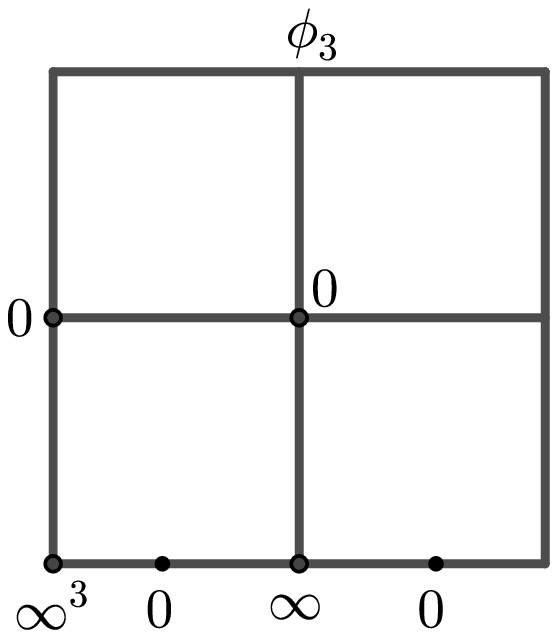}}
\caption{}
\label{Figura1}
\end{figure}

The degree of the Gauss map $g$ equals $3$. Thus, the total curvature

\begin{equation}\label{CT2}
C_T(S):=\int_M K \, dA = -4 \pi \cdot \text{degree}(g) =-12\pi.
\end{equation}
From \eqref{WR}, \eqref{EDPW} and \eqref{WR-WW} we have
\begin{eqnarray*}
\phi_1&=&\wp \, dz - \frac{c^2}{4}\frac{(\wp+e_1)(\wp - \lambda)^2}{\wp-e_1}\,dz,\\
      &=&\wp \, dz - \frac{c^2}{4}\left(\frac{\wp^3}{\wp -e_1}+(e_1-2\lambda)\frac{\wp^2}{\wp-e_1}+(\lambda^2-2e_1\lambda)\frac{\wp}{\wp-e_1}+
			\frac{e_1\lambda^2}{\wp-e_1}\right) \,dz,\\ 
\phi_2&=&i\, \wp \, dz + i\, \frac{c^2}{4}\frac{(\wp+e_1)(\wp - \lambda)^2}{\wp-e_1}\,dz,\\
      &=&i  \wp \, dz + i  \frac{c^2}{4}\left(\frac{\wp^3}{\wp -e_1}+(e_1-2\lambda)\frac{\wp^2}{\wp-e_1}+(\lambda^2-2e_1\lambda)\frac{\wp}{\wp-e_1}+
			\frac{e_1\lambda^2}{\wp-e_1}\right) \,dz,\\ 
\phi_3&=&\frac{c}{2}\wp' \left( \frac{\wp -\lambda}{\wp-e_1}\right)\, dz,\\
      &=&\frac{c}{2} \left( \wp' + (e_1-\lambda)\frac{\wp'}{\wp-e_1}\right)\, dz.
\end{eqnarray*}
Using the Lemma \ref{lemma4} we have
\begin{equation*}
\int \phi_1=-\zeta(z)-\frac{c^2}{4}\left[\frac{\wp'}{6}+\left(\frac{4}{3}e_1^2 - 2e_1\lambda\right)z+(2\lambda-2e_1)\zeta(z)+\left(2\lambda-e_1-\frac{\lambda^2}{e_1} \right)\zeta(z-1/2) \right],\\
\end{equation*}
\begin{equation*}
\int \phi_2=-i\zeta(z)+i\frac{c^2}{4}\left[\frac{\wp'}{6}+\left(\frac{4}{3}e_1^2 - 2e_1\lambda\right)z+(2\lambda-2e_1)\zeta(z)+\left(2\lambda-e_1-\frac{\lambda^2}{e_1} \right)\zeta(z-1/2) \right],
\end{equation*}
\begin{equation*}
\int \phi_3=\frac{c}{2}\Big[\wp + (e_1 - \lambda) \ln |\wp - e_1| \Big].
\end{equation*}
Since $\displaystyle{\int\phi_3}$ is periodic, then 

$$\int_{\alpha} \phi_3=\int_{\beta} \phi_3 =0.$$

From Lemma \ref{lemma6}, it follows that
\begin{equation*}
\text{Re} \int_{\beta} \phi_1= Re\int_{\alpha} \phi_2 =0
\end{equation*}
and
\begin{equation*}
\int_{\alpha} \phi_1= -\pi - \frac{c^2}{4}\left[0+\left(\frac{4}{3}e_1^2 - 2e_1\lambda\right)\cdot 1+(2\lambda-2e_1)\pi+\left(2\lambda-e_1-\frac{\lambda^2}{e_1} \right)\pi \right]
\end{equation*}

\begin{equation}\label{P1}
\int_{\alpha} \phi_1= -\pi - \frac{c^2}{4}\left[\left(\frac{4}{3}e_1^2 - 3e_1\pi\right)+(4\pi-2e_1)\lambda-\frac{\pi}{e_1}\lambda^2 \right],
\end{equation}

\begin{equation*}
\int_{\beta} \phi_2= -i(-i \pi)+i\frac{c^2}{4}\left[0+\left(\frac{4}{3}e_1^2 - 2e_1\lambda\right)\cdot i+(2\lambda-2e_1)(-i \pi)+\left(2\lambda-e_1-\frac{\lambda^2}{e_1} \right)(-i\pi) \right],
\end{equation*}

\begin{equation*}
\int_{\beta} \phi_2= -\pi-\frac{c^2}{4}\left[\left(\frac{4}{3}e_1^2 - 2e_1\lambda\right)+(2\lambda-2e_1)(-\pi)+\left(2\lambda-e_1-\frac{\lambda^2}{e_1} \right)(-\pi) \right],
\end{equation*}

\begin{equation}\label{P2}
\int_{\beta} \phi_2= -\pi-\frac{c^2}{4}\left[\left(\frac{4}{3}e_1^2 + 3e_1\pi\right)-(4\pi+2e_1)\lambda+\frac{\pi}{e_1}\lambda^2 \right].
\end{equation}
Therefore \eqref{PP1} holds if, and only if
\begin{equation*}
\ -\pi - \frac{c^2}{4}\left[\left(\frac{4}{3}e_1^2 - 3e_1\pi\right)+(4\pi-2e_1)\lambda-\frac{\pi}{e_1}\lambda^2 \right]=0,
\end{equation*}
\begin{equation*}
 -\pi-\frac{c^2}{4}\left[\left(\frac{4}{3}e_1^2 + 3e_1\pi\right)-(4\pi+2e_1)\lambda+\frac{\pi}{e_1}\lambda^2 \right]=0.
\end{equation*}
This implies that
\begin{equation*}
\left(\frac{4}{3}e_1^2 + 3e_1\pi\right)-(4\pi+2e_1)\lambda+\frac{\pi}{e_1}\lambda^2=
\left(\frac{4}{3}e_1^2 - 3e_1\pi\right)+(4\pi-2e_1)\lambda-\frac{\pi}{e_1}\lambda^2.
\end{equation*}
This is equivalent to
\begin{equation}\label{Eq_Sol_P1}
\lambda^2-4e_1\lambda+3e_1^2=0.
\end{equation}
The roots of equation \eqref{Eq_Sol_P1} are $\lambda=3e_1$ and $\lambda =e_1.$ By hypothesis, $\lambda \neq e_1$ and thus we obtain

\begin{equation}\label{DW-WW}
\left\{
\begin{array}{ll}
	 g &= c \, \displaystyle{\frac{(\wp + e_1)(\wp-3e_1)}{\wp'}},  \ \ c = \frac{1}{e_1}\sqrt{\frac{6\pi}{7}}.\\
	\mathbf \eta&=2\wp \, dz.
	\end{array}
	\right.
	\end{equation}
	
Now, we will show that M has no real period around $p_1=\pi(1/2)$ and $p_2=\pi(0)$. To do this we will compute residues at the points $p_1$ and $p_2.$ The functions $\phi_1$ and $\phi_2$ have poles of order two at $p_1, \ p_2$  and have no residues, while the function $\phi_3$ has a simple pole at $p_1$. By evaluating the residues of the elliptic function $\phi_3$ at $p_1$, we have
\begin{equation*}
\text{Res}_{p_1} \phi_3= -\frac{c}{2e_1}\wp''(1/2)=-2c\, e_1 \in \mathbb{R}.
\end{equation*}
As the sum of all the residues of an elliptic function at the poles inside $F$ is zero, then $\text{Res}_{p_2} \phi_3$ is also real.

At the points $p_1$ and $p_2$,  $ds=\frac{1}{2}(1+|g|^2)|\eta|$ has a pole of order at least $2$. Thus,
\begin{equation}
\int_{\ell} (1+|g|^2)|\eta|=\infty,
\end{equation}
for any divergent curve $\ell$, around $p_1$ and $p_2$. By Proposition \ref{Prop1}, we obtain a complete minimal surface $S=X(M)$
with genus one and two ends (Enneper - Catenoid), where
\begin{align*}
X_1&=\text{Re}\int_{\frac{1+i}{2}}^z\phi_1=\text{Re}\left\{-\zeta(z)-\frac{c^2}{4}\left[\frac{\wp'(z)}{6}-\frac{14}{3}e_1^2\,z+4e_1\zeta(z)-4e_1\zeta(z-1/2) \right]\right\}+\frac{3\pi^2}{7e_1},\\
X_2&=\text{Re}\int_{\frac{1+i}{2}}^z \phi_2=\text{Re}\left\{-i\zeta(z)+i\frac{c^2}{4}\left[\frac{\wp'(z)}{6}-\frac{14}{3}e_1^2\,z+4e_1\zeta(z)-4e_1\zeta(z-1/2) \right]\right\},\\
X_3&=\text{Re}\int_{\frac{1+i}{2}}^z \phi_3=\text{Re}\left\{\frac{c}{2}\Big[\wp -2e_1 \, \ln |\wp - e_1| \Big]\right\}+\sqrt{\frac{6\pi}{7}}\ln e_1, \ \ c = \frac{1}{e_1}\sqrt{\frac{6\pi}{7}}.
\end{align*}
This completes the proof.
\end{proof}

\begin{remark}
If $\lambda=e_1,$ then 
\begin{equation*}
\left\{
\begin{array}{ll}
	 g &= c \, \displaystyle{\frac{(\wp + e_1)(\wp-e_1)}{\wp'}}=\frac{c}{4}\frac{\wp'}{\wp},  \ \ \frac{c}{4} =\frac{1}{2e_1}\sqrt{\frac{3\pi}{2}},\\
	\mathbf \eta&=2\wp \, dz.
	\end{array}
	\right.
	\end{equation*}
But, this is the Weierstrass data of Chen-Gackstatter surface.
\end{remark}

\section{Proof of the Theorem}

In order to find a new example of a minimal surface with three ends, being two catenoid-type ends and one  Ennerper-type end,  we will add one more catenoid-type end to minimal surface described in Proposition \ref{S_VW}, by symmetry.

Let $T^2=\mathbb{C}/\mathcal{L}$ be the o torus with complex structure induce by canonical projection $\pi: \mathcal{C} \rightarrow T^2.$ 
Let $M=T^2 -\{p_1, p_2,p_3\}$, where
\begin{equation}
p_1=\pi(1/2),\ p_2=\pi(0) \ \text{and} \ p_3=\pi(i/2).
\end{equation}
The Weierstrass data  $(g, \eta)$ is given by

\begin{equation}\label{WR-V}
\left\{
\begin{array}{ll}
	 g &= c \, \displaystyle{\frac{(\wp - 3e_1)(\wp+\lambda)}{\wp'}},  \ \ c >0, \ \lambda> 0, \ \lambda \neq e_1,\\
	\mathbf \eta&=2\wp \, dz.
	\end{array}
	\right.
\end{equation}
The Figure \ref{Figura2} shows the zeros and poles of $g$, $\eta$ and $\phi_3$ on $F.$
\begin{figure}[h]
\subfigure[Zeros and polos of $g$\label{Figure1}]{
\centering
\includegraphics[totalheight=3cm]{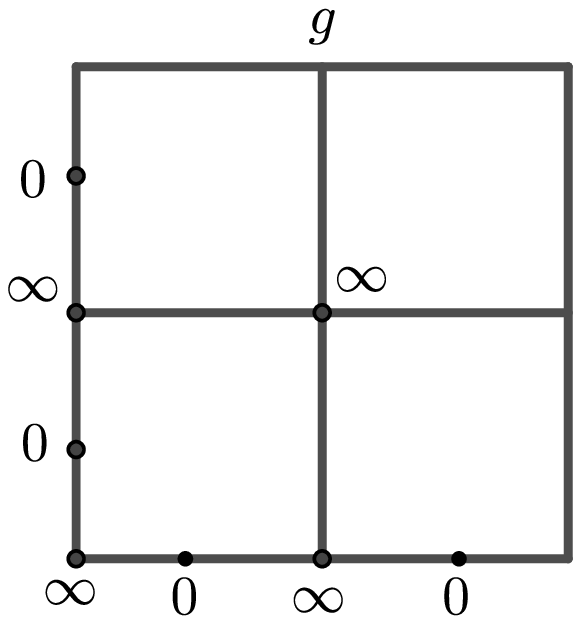}}
\subfigure[Zeros and polos of $\eta$\label{Figure2}]{
\includegraphics[totalheight=3cm]{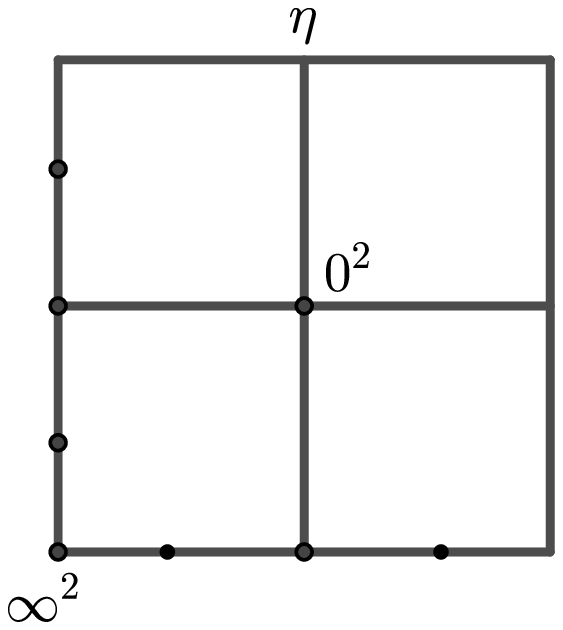}}
\subfigure[Zeros and polos of $\phi_3$\label{Figure3}]{
\includegraphics[totalheight=3cm]{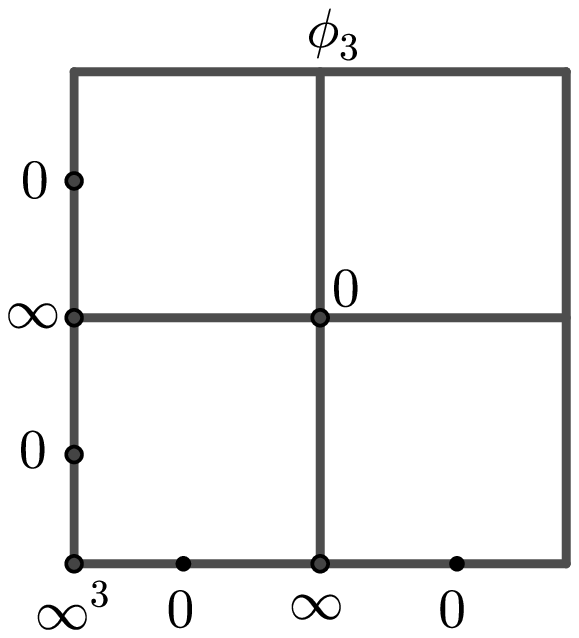}}
\caption{}
\label{Figura2}
\end{figure}

From \eqref{WR}, \eqref{EDPW} and \eqref{WR-V} we have
\begin{equation}\label{Phi1V}
\begin{aligned}
\phi_1&=\wp \, dz - \frac{c^2}{4}\cdot\frac{(\wp-3e_1)^2(\wp + \lambda)^2}{(\wp-e_1)(\wp+e_1)}\,dz,\\
\phi_1&=\wp \, dz - \frac{c^2}{8e_1}\left[\left(\frac{\wp^4}{\wp-e_1}-\frac{\wp^4}{\wp+e_1}\right)+(2\lambda-6e_1)\left(\frac{\wp^3}{\wp-e_1}-
\frac{\wp^3}{\wp+e_1}\right)+\right.\\
&+(\lambda^2-12e_1\lambda+9e_1^2)\left(\frac{\wp^2}{\wp-e_1}-\frac{\wp^2}{\wp+e_1}\right)+(18e_1^2\lambda-6e_1\lambda^2)\left(\frac{\wp}{\wp-e_1}-
\frac{\wp}{\wp+e_1}\right)+\\
&+\left.9e_1^2\lambda^2\left(\frac{1}{\wp-e_1}-\frac{1}{\wp+e_1}\right)\right]\, dz,\\ 
\end{aligned}
\end{equation}

\begin{equation}\label{Phi2V}
\begin{aligned}
\phi_2&=i\wp \, dz +i \frac{c^2}{4}\cdot\frac{(\wp-3e_1)^2(\wp + \lambda)^2}{(\wp-e_1)(\wp+e_1)}\,dz,\\
\phi_2&=i\wp \, dz +i \frac{c^2}{8e_1}\left[\left(\frac{\wp^4}{\wp-e_1}-\frac{\wp^4}{\wp+e_1}\right)+(2\lambda-6e_1)\left(\frac{\wp^3}{\wp-e_1}-
\frac{\wp^3}{\wp+e_1}\right)+\right.\\
&+(\lambda^2-12e_1\lambda+9e_1^2)\left(\frac{\wp^2}{\wp-e_1}-\frac{\wp^2}{\wp+e_1}\right)+(18e_1^2\lambda-6e_1\lambda^2)\left(\frac{\wp}{\wp-e_1}-
\frac{\wp}{\wp+e_1}\right)+\\
&+\left.9e_1^2\lambda^2\left(\frac{1}{\wp-e_1}-\frac{1}{\wp+e_1}\right)\right]\, dz,\\
\end{aligned}
\end{equation}
\begin{eqnarray*}
\phi_3&=&2c \frac{(\wp-3e_1)(\wp+\lambda) \wp}{\wp'} \,dz=\frac{c}{2}\wp'\frac{(\wp-3e_1)(\wp+\lambda)}{(\wp-e_1)(\wp+e_1)}\, dz,\\
 \phi_3&=&\frac{c}{4e_1}\left[\left(\frac{\wp' \wp^2}{\wp-e_1}-\frac{\wp'\wp^2}{\wp+e_1}\right)+
			(\lambda-3e_1)\left(\frac{\wp' \wp}{\wp-e_1}-\frac{\wp'\wp}{\wp+e_1}\right)-3e_1\lambda 
			\left(\frac{\wp'}{\wp-e_1}-\frac{\wp'}{\wp+e_1}\right) \right]\, dz.
\end{eqnarray*}
From \eqref{EqW6} and \eqref{EqW7} we have
\begin{equation}\label{Phi3V}
\phi_3=\frac{c}{4e_1}\left[2e_1 \wp' - (2e_1^2+2e_1\lambda)\frac{\wp'}{\wp-e_1}+(4e_1\lambda-4e_1^2)\frac{\wp'}{\wp+e_1}\right]\, dz.
\end{equation}
Using the Lemma \ref{lemma5} we have
\begin{eqnarray}\label{IPhi1V}
\int \phi_1&=&-\zeta(z)-\frac{c^2}{8e_1}\left[\left(\frac{e_1}{3}\wp'+\frac{5}{3}e_1^3 z-\frac{e_1^2}{2}\zeta(z-1/2)+\frac{e_1^2}{2}\zeta(z-i/2)\right)\right.+\nonumber\\
&+&(2\lambda-6e_1)\left(-2e_1\zeta(z)-\frac{e_1}{2}\zeta(z-1/2)-\frac{e_1}{2}\zeta(z-i/2)\right)+\nonumber\\
&+&\left(\lambda^2-12e_1\lambda+9e_1^2\right)\left(e_1 z -\frac{1}{2}\zeta(z-1/2)+\frac{1}{2}\zeta(z-i/2)\right)+\\
&+&\left(18e_1^2\lambda-6e_1\lambda^2 \right)\left(-\frac{1}{2e_1}\zeta(z-1/2)-\frac{1}{2e_1}\zeta(z-i/2) \right)+\nonumber\\
&+&\left.9e_1^2\lambda^2 \left(-\frac{1}{2e_1^2}\zeta(z-1/2)+\frac{1}{2e_1^2}\zeta(z-i/2)-\frac{z}{e_1} \right)\right]-\text{constant},\nonumber
\end{eqnarray}

\begin{eqnarray}\label{IPhi2V}
\int \phi_2&=&-i\zeta(z)+i\frac{c^2}{8e_1}\left[\left(\frac{e_1}{3}\wp'+\frac{5}{3}e_1^3 z-\frac{e_1^2}{2}\zeta(z-1/2)+\frac{e_1^2}{2}\zeta(z-i/2)\right)\right.+\nonumber\\
&+&(2\lambda-6e_1)\left(-2e_1\zeta(z)-\frac{e_1}{2}\zeta(z-1/2)-\frac{e_1}{2}\zeta(z-i/2)\right)+\nonumber\\
&+&\left(\lambda^2-12e_1\lambda+9e_1^2\right)\left(e_1 z -\frac{1}{2}\zeta(z-1/2)+\frac{1}{2}\zeta(z-i/2)\right)+\\
&+&\left(18e_1^2\lambda-6e_1\lambda^2 \right)\left(-\frac{1}{2e_1}\zeta(z-1/2)-\frac{1}{2e_1}\zeta(z-i/2) \right)+\nonumber\\
&+&\left.9e_1^2\lambda^2 \left(-\frac{1}{2e_1^2}\zeta(z-1/2)+\frac{1}{2e_1^2}\zeta(z-i/2)-\frac{z}{e_1} \right)\right]-\text{constant},\nonumber
\end{eqnarray}

\begin{equation}
\int \phi_3=\frac{c}{4e_1}\Big[2e_1 \wp - (2e_1^2+2e_1\lambda)\ln |\wp-e_1|+(4e_1\lambda-4e_1^2)\ln|\wp+e_1|\Big]-\text{constant}.
\end{equation}
Since $\displaystyle{\int\phi_3}$ is periodic, then 

$$\int_{\alpha} \phi_3=\int_{\beta} \phi_3 =0.$$
From Lemma \ref{lemma6}, it follows that
\begin{equation*}
\text{Re} \int_{\beta} \phi_1= \text{Re}\int_{\alpha} \phi_2 =0
\end{equation*}
and
\begin{eqnarray*}
\int_{\alpha} \phi_1= -\pi &-&\frac{c^2}{8e_1}\left[\frac{5}{3}e_1^3 + (2\lambda-6e_1)(-3e_1\pi)+\left(\lambda^2-12e_1\lambda+9e_1^2 \right)e_1+\right.\\
&+&\left.(18e_1^2\lambda-6e_1\lambda^2)\left(\frac{-\pi}{e_1}\right)+9e_1^2\lambda^2\left(\frac{-1}{e_1}\right)\right],
\end{eqnarray*}

\begin{eqnarray*}
\int_{\beta} \phi_2= -\pi &-& \frac{c^2}{8e_1}\left[\frac{5}{3}e_1^3 + (2\lambda-6e_1)(3e_1\pi)+\left(\lambda^2-12e_1\lambda+9e_1^2 \right)e_1+\right.\\
&+&\left.(18e_1^2\lambda-6e_1\lambda^2)\left(\frac{\pi}{e_1}\right)+9e_1^2\lambda^2\left(\frac{-1}{e_1}\right)\right].
\end{eqnarray*}
Therefore,
\begin{equation*}
\int_{\alpha} \phi_1= -\pi -\frac{c^2}{8e_1}\left[\left(\frac{5}{3}e_1^3+18e_1^2\pi+9e_1^3 \right)+\left(-24e_1\pi-12e_1^2 \right)\lambda+(-8e_1+6\pi)\lambda^2\right]
\end{equation*}
and
\begin{equation*}
\int_{\beta} \phi_2= -\pi -\frac{c^2}{8e_1}\left[\left(\frac{5}{3}e_1^3-18e_1^2\pi+9e_1^3 \right)+\left(24e_1\pi-12e_1^2 \right)\lambda+(-8e_1-6\pi)\lambda^2\right].
\end{equation*}
Thus, $\displaystyle{\text{Re} \int_{\alpha} \phi_1= \text{Re}\int_{\beta} \phi_2 =0},$ if and only if there exists a real number $\lambda$ such that
\begin{equation}\label{Eq2G}
\lambda^2-4e_1\lambda+3e_1^2=0 
\end{equation}
and
\begin{equation}\label{Eqc}
 c=\sqrt{\frac{6\pi}{33e_1\lambda-26e_1^2}}.
\end{equation}
The roots of the equation \eqref{Eq2G} are $\lambda_1=e_1$ and $\lambda_2 = 3e_2.$ By hypothesis $\lambda \neq e_1,$ then 
\begin{equation}\label{Eqc}
 \lambda=3e_1 \ \ \Rightarrow \ \  c=\frac{1}{e_1}\sqrt{\frac{6\pi}{73}}.
\end{equation}
Thus, we obtain 
\begin{equation}\label{DW-V}
\left\{
\begin{array}{ll}
	 g &= c \, \displaystyle{\frac{(\wp - 3e_1)(\wp+3e_1)}{\wp'}},  \ \ c = \frac{1}{e_1}\sqrt{\frac{6\pi}{73}}.\\
	\mathbf \eta&=2\wp \, dz.
	\end{array}
	\right.
	\end{equation}
The Weierstrass data of the above form, produce the symmetry of the surface that we require. The degree of the Gauss map $g$ equals $4$. Thus, the total curvature is $-16\pi.$

Now, we will show that M has no real period around $p_1=\pi(1/2),$ $p_2=\pi(0)$ and $p_3=\pi(i/2)$. To do this we will compute residues at points $p_1, p_2$ and $p_3.$ The functions $\phi_1$ and $\phi_2$ have poles of order two at $p_1, \ p_2$ and $p_3$ and have no residues, while the function $\phi_3$ has a simple pole at $p_1$ and $p_3$. Evaluating the residues of the elliptic function $\phi_3$ at $p_1$ and $p_3$, we have
\begin{equation*}
\text{Res}_{p_1} \phi_3= -\frac{c}{e_1}\wp''(1/2)=-4c\, e_1 \in \mathbb{R},
\end{equation*}
\begin{equation*}
\text{Res}_{p_3} \phi_3= \frac{c}{e_1}\wp''(i/2)=4c\, e_1 \in \mathbb{R},
\end{equation*}
This implies that $\text{Res}_{p_2} \phi_3=0.$

At $p_1, p_2$ and $p_3$,  $ds=\frac{1}{2}(1+|g|^2)|\eta|$ has a pole of order at least $2$. Thus,
\begin{equation}
\int_{\ell} (1+|g|^2)|\eta|=\infty,
\end{equation}
for any divergent curve $\ell$, around $p_1, \ p_2$ and $p_3$. Therefore, by Proposition \ref{Prop1}, we obtain a complete minimal surface $S=X(M)$
with genus one, two catenoid-type ends and one Enneper-type end, where $X(z)=(X_1(z),X_2(z),X_3(z)),$   $X(1/2,1/2)=(0,0,0),$
\begin{equation*}
\begin{aligned}
X_1&=\text{Re}\left\{-\zeta(z)-\frac{c^2}{4}\left[\frac{\wp'(z)}{6}+16e_1\zeta(z-i/2)-16e_1\zeta(z-1/2)-\frac{146}{3}e_1^2\,z \right]\right\}+\frac{12\pi^2}{73e_1},\\
X_2&=\text{Re}\left\{-i\zeta(z)+i\frac{c^2}{4}\left[\frac{\wp'(z)}{6}+16e_1\zeta(z-i/2)-16e_1\zeta(z-1/2)-\frac{146}{3}e_1^2\,z \right]\right\}+\frac{12\pi^2}{73e_1},\\
X_3&=\text{Re}\left\{\frac{c}{2}\left[\wp +4e_1  \ln \left|\frac{\wp+e_1}{\wp - e_1}\right| \right]\right\}, \ \ c = \frac{1}{e_1}\sqrt{\frac{6\pi}{73}}.
\end{aligned}
\end{equation*}

Now, we will study the symmetries of surface S=$X(M)$. To do this we will define on $F$ the curves:
\begin{equation*}
\begin{array}{lll}
&\zeta_1(u)=u, \ 0<u<1/2, \ \ &\zeta_2(u)=u, \ 1/2<u<1,\\
&\zeta_3(u)=\frac{i}{2}+u, \ 0<u<1, \ \ &\zeta_4(u)=iu, \ 0<u<1/2,\\
&\zeta_5(u)=iu, \ 1/2<u<1, \ &\zeta_6(u)=\frac{1}{2}+iu, \ 0<u<1,\\
&\zeta_7(u)=u+i(1-u), \ 0<u<1, \ &\zeta_8(u)=u+iu, \ 0<u<1.
\end{array}
\end{equation*}
From Lemma \ref{Karcher.1989} and Proposition \ref{Wohlgemuth.1991}, we have that $\gamma_j:=X\circ \zeta_j, $ $j=1,\ldots,6,$ are planar geodesics and $\gamma_j:=X\circ \zeta_j,$ $j=7,8,$ are two straight lines cross at 0 and at the Enneper end. We can easily show that $\gamma_j$,  $j=1,2,3,$ are contained in the $(x_1,x_3)$-plane, $\gamma_j$, $j=4,5,6,$ are contained in the $(x_2,x_3)$-plane and $\gamma_j, \ j=7,8,$ are contained in the lines $x_1\pm x_2=x_3=0.$ The Schwarz reflection principle for minimal surfaces, implies that surface $S$ has the $(x_1, x_3)$-plane and the $(x_2, x_3)$-plane as reflective planes of symmetry, and the surface is invariant under rotation by $\pi$ about the lines $x_1\pm x_2=x_3=0$.

Finally, substituting \eqref{Eqc} into \eqref{Phi1V}, \eqref{Phi2V} and \eqref{Phi3V}. From Lemma \ref{LemaS1}, we have the following properties:
\begin{equation}\label{refl1}
X(\beta(w_2+z))=( X_1, X_2,X_3)(\beta(w_2+z))=(X_1,-X_2,X_3)(w_2+z), 
\end{equation}
\begin{equation}\label{rot1}
X(\rho(w_2+z))=( X_1, X_2,X_3)(\rho(w_2+z))=( -X_2, X_1,-X_3)(w_2+z). 
\end{equation}
Let $G:=\langle \beta, \rho \rangle$ be the dihedral group with 8 elements. According to \eqref{refl1} and \eqref{rot1}, we may consider $G$ as acting on $S=X(M)$, 
$A_{\sigma}:G \times S \rightarrow S$, $A_{\sigma}(\sigma,X(z)):=X(\sigma(z))$,  by identifying the generators $\beta$ and $\rho$ with the orthogonal motions
\[
	A_{\beta}=\left[
	\begin{array}{ccc}
	1&0&0\\
	0&-1&0\\
	0&0&1
	\end{array}
	\right], \ \ \ 
	A_{\rho}=\left[
	\begin{array}{ccc}
	0&-1&0\\
	1&0&0\\
	0&0&-1
	\end{array}
	\right].
	\]
Therefore, the symmetry group of $S$ is $\{Id, A_{\beta}, A_{\rho}, A_{\rho}A_{\beta},  A_{\rho}^2A_{\beta}, A_{\rho}^3A_{\beta},A_{\rho}^2,A_{\rho}^3 \}.$ This completes the proof of the Theorem.

\begin{remark}
We note that if $\lambda=e_1,$ then 
\begin{equation*}
\left\{
\begin{array}{ll}
	 g &= c \, \displaystyle{\frac{(\wp - 3e_1)(\wp+e_1)}{\wp'}},  \ \ c=\frac{1}{e_1}\sqrt{\frac{6\pi}{7}}.\\
	\mathbf \eta&=2\wp \, dz.
	\end{array}
	\right.
	\end{equation*}
But, this is exactly the Weierstrass data of minimal surface described in Proposition \ref{S_VW}.
\end{remark}

\begin{remark}
From the Gackstatter \cite{Gackstatter.1976}, Jorge-Meeks \cite{Jorge.1983} formula, 
\begin{equation*}
C_T(S)=2\pi (2-2\mathbf{g}-N-\sum_{\nu=1}^N k_\nu), \ \text{where} \ k_{\nu} \ \text{is the order of the end,}
\end{equation*}
we see that the order of Enneper end is $3$, thus the minimal surface $S$ as described in Theorem \ref{Teorema1} is not embedded. I am working in a similar way to \cite{Hoffman.1985} to show that there exists a compact set $K \subset \mathbb{R}^3$ such that a region of $S$ contained in $K$ is embedded.  
\end{remark}

\begin{remark}
During the writing of this paper, I found some computer graphic pictures (see \cite{Weber.2019})  of a surface belonging to a one-parameter
family of minimal surfaces with two catenoid ends and one Enneper end.  
\end{remark}

\begin{question}
Are the surfaces in Proposition \ref{S_VW} and Theorem \ref{Teorema1} the only complete minimal surface with genus one or two?
\end{question}

\begin{question}
Is it possible to construct examples of complete minimal surfaces with genus $\mathbf{g}>1$ with two catenoid ends and one Enneper end?
\end{question}

\begin{question}
Is it possible to construct examples of complete minimal surfaces $S$ with genus $\mathbf{g}>1$ increasing dihedral symmetry in such a way that $S$ 
have two catenoid ends and one higher order Enneper end?
\end{question}

\bibliographystyle{amsplain} 
\bibliography{BiblGeometria}

\providecommand{\bysame}{\leavevmode\hbox to3em{\hrulefill}\thinspace}
\providecommand{\MR}{\relax\ifhmode\unskip\space\fi MR }
\providecommand{\MRhref}[2]{%
  \href{http://www.ams.org/mathscinet-getitem?mr=#1}{#2}
}
\providecommand{\href}[2]{#2}
\begin{thebibliography}{10}

\bibitem{Gray.2017}
Elsa Abbena, Simon Salamon, and Alfred Gray, \emph{Modern differential geometry
  of curves and surfaces with mathematica}, Chapman and Hall/CRC, 2017.

\bibitem{Chand.1985}
Komaravolu Chandrasekharan, \emph{Elliptic functions}, Springer-Verlag, 1985.

\bibitem{Chen.1982}
Chi~Cheng Chen and Fritz Gackstatter, \emph{Elliptische und hyperelliptische
  {F}unktionen und vollst\"{a}ndige {M}inimalfl\"{a}chen vom {E}nneperschen
  {T}yp}, Math. Ann. \textbf{259} (1982), no.~3, 359--369. \MR{661204}

\bibitem{Costa.1984}
Celso~J. Costa, \emph{Example of a complete minimal immersion in
  {$\mathbf{R}^3$} of genus one and three embedded ends}, Bol. Soc. Brasil.
  Mat. \textbf{15} (1984), no.~1-2, 47--54. \MR{794728}

\bibitem{Costa.1989}
\bysame, \emph{Uniqueness of minimal surfaces embedded in {$\mathbf{R}^3$} with
  total curvature {$12\pi$}}, J. Differential Geom. \textbf{30} (1989), no.~3,
  597--618. \MR{1021368}

\bibitem{Nedir.1994}
Nedir Do~Espirito-Santo, \emph{Complete minimal surfaces in {$\mathbf{R}^3$}
  with type {E}nneper end}, Ann. Inst. Fourier (Grenoble) \textbf{44} (1994),
  no.~2, 525--557. \MR{1296742}

\bibitem{Fang.1990}
Yi~Fang, \emph{A new family of {E}nneper type minimal surfaces}, Proc. Amer.
  Math. Soc. \textbf{108} (1990), no.~4, 993--1000. \MR{1012931}

\bibitem{Fujimori.2016}
Shoichi Fujimori and Toshihiro Shoda, \emph{Minimal surfaces with two ends
  which have the least total absolute curvature}, Pacific J. Math. \textbf{282}
  (2016), no.~1, 107--144. \MR{3463426}

\bibitem{Fujimoto.2013}
Hirotaka Fujimoto, S~Hildebrandt, D~Hoffmann, H~Karcher, and L~Simon,
  \emph{Geometry v: Minimal surfaces}, vol.~90, Springer Science \& Business
  Media, 2013.

\bibitem{Gackstatter.1976}
Fritz Gackstatter, \emph{{\"U}ber die dimension einer minimalfl{\"a}che und zur
  ungleichung von st. cohn-vossen}, Archive for Rational Mechanics and Analysis
  \textbf{61} (1976), no.~2, 141--152.

\bibitem{Hoffman.1990}
David Hoffman and William~H. Meeks, III, \emph{Embedded minimal surfaces of
  finite topology}, Ann. of Math. (2) \textbf{131} (1990), no.~1, 1--34.
  \MR{1038356}

\bibitem{Hoffman.1985}
David~A. Hoffman and William Meeks, III, \emph{A complete embedded minimal
  surface in {$\mathbf{R}^3$} with genus one and three ends}, J. Differential
  Geom. \textbf{21} (1985), no.~1, 109--127. \MR{806705}

\bibitem{Jorge.1983}
Luqu{\'e}sio~P Jorge and William~H Meeks~III, \emph{The topology of complete
  minimal surfaces of finite total gaussian curvature}, Topology \textbf{22}
  (1983), no.~2, 203--221.

\bibitem{Kang.2003}
Jian~Ling Kang and Hong Wang, \emph{A family of complete immersed minimal
  surfaces with only one end}, J. Zhejiang Univ. Sci. Ed. \textbf{30} (2003),
  no.~6, 612--616. \MR{2029983}

\bibitem{Karcher.1989}
Hermann Karcher, \emph{Construction of minimal surfaces, surveys in geometry},
  University of Tokyo, 1989.

\bibitem{Osserman.2013}
Robert Osserman, \emph{A survey of minimal surfaces}, Courier Corporation,
  2013.

\bibitem{Thayer.1995}
Edward~C. Thayer, \emph{Higher-genus {C}hen-{G}ackstatter surfaces and the
  {W}eierstrass representation for surfaces of infinite genus}, Experiment.
  Math. \textbf{4} (1995), no.~1, 19--39. \MR{1359415}

\bibitem{Weber.2015}
Matthias Weber, \emph{Virtual minimal surface museum,
  http://www.indiana.edu/~minimal},  (2015).

\bibitem{Weber.2019}
\bysame, \emph{Virtual minimal surface museum,
  https://minimalsurfaces.blog/author/ \\matthiasweber64/},  (2019).

\bibitem{Wohlgemuth.1991}
Meinhard Wohlgemuth, \emph{Higher genus minimal surfaces by growing handles out
  of a catenoid}, manuscripta mathematica \textbf{70} (1991), no.~1, 397--428.

\end{thebibliography}

\end{document}